\documentclass[11pt]{article}

\usepackage[margin=1in]{geometry}
\usepackage{amsmath,amssymb,amsthm,mathtools}
\usepackage{enumitem}
\usepackage[colorlinks=true,linkcolor=blue,citecolor=blue,urlcolor=blue]{hyperref}

\numberwithin{equation}{section}

\newtheorem{theorem}{Theorem}[section]
\newtheorem{lemma}[theorem]{Lemma}
\newtheorem{proposition}[theorem]{Proposition}

\newtheorem{remark}[theorem]{Remark}

\newcommand{\Om}{\Omega}

\newcommand{\R}{\mathbb R}

\newcommand{\supp}{\operatorname{supp}}
\newcommand{\dist}{\operatorname{dist}}
\newcommand{\pnu}{\partial_\nu}

\newcommand{\calF}{\mathcal F}
\newcommand{\calA}{\mathcal A}

\newcommand{\Lam}{\Lambda}

\title{Simultaneous Recovery of the Initial Source and Sound Speed
for the Wave Equation under a Constitutive Constraint}
\author{ 
    \qquad Amir Moradifam\footnote{Department of Mathematics, University of California, Riverside, California, USA. E-mail: amirm@ucr.edu. Amir Moradifam is supported by NSF grant DMS-1953620.}
}
\date{}

\begin{document}
	\maketitle

\begin{abstract}
We study the simultaneous recovery of the initial source and sound speed for the scalar wave equation from a single boundary measurement. Although the recovery of either parameter separately is well understood under suitable geometric hypotheses, simultaneous recovery remains open in general, and stable recovery is further obstructed by the inherent instability of the linearized problem. We show that both uniqueness and stability can be obtained when the two unknowns are coupled through a prescribed constitutive relation arising from a common underlying material. Under a quantitative nondegeneracy condition, motivated by calibrated material regimes, the coupled inverse problem reduces to a single inverse source problem. Applying the microlocal and Carleman framework of Stefanov and Uhlmann to this reduced problem, we establish global uniqueness and Lipschitz stability under geometric conditions involving strictly convex foliations and geodesic visibility. We also obtain partial-data results, proving local uniqueness and Lipschitz stability in compact visible regions where the constitutive nondegeneracy condition holds.

\end{abstract}

\section{Introduction}
Inverse problems for wave equations arise when waves generated inside a medium are observed on the boundary and one seeks to recover interior information about the medium or the source. In this paper we consider the scalar wave equation

\begin{equation}
\label{eq:intro_wave}
\begin{aligned}
(\partial_t^2-c^2(x)\Delta)u &= 0,
        && (t,x)\in (0,T)\times\R^n,\\
u(0,x) &= f(x),
        && x\in\R^n,\\
\partial_tu(0,x) &= 0,
        && x\in\R^n.
\end{aligned}
\end{equation}

Here $u=u(t,x)$ is the wave field, $c(x)>0$ is the sound speed, and $f(x)$ is the initial source. If $\Om\subset\R^n$ is the region of interest, the measured data are the boundary trace
\[
\Lam_{f,c}=u|_{(0,T)\times\partial\Om}.
\]
The inverse problem is to recover the initial source, the sound speed, or both, from this boundary measurement. Throughout the paper, the speed is assumed to be known outside $\Om$.

Inverse problems of this type arise, in particular, in photoacoustic and thermoacoustic tomography. In these imaging modalities, electromagnetic energy is deposited in a medium and converted, through thermoelastic expansion, into an acoustic wave. When the energy deposition occurs on a time scale much shorter than the subsequent acoustic propagation, it is represented by the initial source $f$, while the resulting wave field propagates according to the homogeneous wave equation \eqref{eq:intro_wave}. These modalities have important applications in biomedical imaging and have been studied extensively from both mathematical and experimental points of view \cite{HaltmeierScherzerBurgholzerPaltauf2004,KrugerKiserReineckeKruger2003,KrugerReineckeKruger1999,KuchmentKunyansky2008,Wang2009,XuWang2006}. Related photoacoustic models also arise in gas sensing \cite{Elia2009,Palzer2020}. Closely related scalar models occur in laser ultrasonics and nondestructive testing, where a short laser pulse generates an elastic or acoustic wave in a solid specimen \cite{ScrubyDrain1990,ZareiPilla2024}.

There is a substantial literature on inverse problems associated with \eqref{eq:intro_wave}, but most results assume that one of the two quantities $f$ or $c$ is already known. When the sound speed is known and non-trapping, the recovery of the initial source is well understood and explicit inversion and reconstruction formulas can be found in \cite{AgranovskyKuchmentKunyansky2009,AmmariBossyJugnonKang2010, FinchRakesh2009,Hristova2009,HristovaKuchmentNguyen2008, KuchmentKunyansky2008,StefanovUhlmann2009,StefanovUhlmann2011,StefanovUhlmann2013SourceSpeed, Wang2009}. Conversely, if the initial source is known, the recovery of a smooth sound speed from one boundary measurement was studied by Stefanov and Uhlmann \cite{StefanovUhlmann2013SourceSpeed}, who proved uniqueness and stability under sharp geometric hypotheses formulated in terms of strictly convex foliations and geodesic visibility. Oksanen and Uhlmann quantified how a modeling error in the sound speed affects known-speed reconstruction and obtained a corresponding stability estimate \cite{OksanenUhlmann2014}.

From the practical point of view, the assumption that the sound speed is known in advance is often unrealistic. In photoacoustic and thermoacoustic imaging, even small errors in the sound speed may produce significant artifacts in the reconstructed source \cite{HristovaKuchmentNguyen2008,JinWang2006,JoseWilleminkEtAl2012}. One approach is to use an auxiliary modality, such as ultrasonic transmission tomography, to estimate the speed before reconstructing the source \cite{JinWang2006,JoseWilleminkEtAl2012}. From both a theoretical and practical point of view, it is therefore natural to ask whether the initial source and the sound speed can be recovered simultaneously from the same boundary measurement.

The simultaneous recovery of both $f$ and $c$ from a single measurement is considerably more delicate, and the known positive results require strong structural assumptions. Finch and Hickmann \cite{FinchHickmann2013} related the problem to the interior transmission eigenvalue problem and obtained range-separation and uniqueness results, in particular for radially symmetric non-trapping sound speeds. Liu and Uhlmann \cite{LiuUhlmann2015} proved simultaneous recovery in special cases, including the recovery of a constant sound speed and a source independent of one spatial variable under an additional positivity assumption. Knox and Moradifam \cite{KnoxMoradifam2020} proved that equality of the boundary measurements implies the integral identity
$$
\int_{\Omega}(c_1^{-2}-c_2^{-2})\varphi\,dx=0
$$
for every harmonic function $\varphi$. This identity does not require any knowledge of the source and yields uniqueness under additional assumptions, for instance if $c_1^{-2}-c_2^{-2}$ is harmonic, or if the two speeds are ordered. Kian and Uhlmann \cite{KianUhlmann2025} proved uniqueness under monotonicity hypotheses for piecewise constant sound speeds; see also the related inclusion-type results in \cite{KianLiu2025,KianTriki2024}.

All known results on simultaneous recovery of the source $f$ and the sound speed $c$ impose restrictive structural assumptions. More importantly, Stefanov and Uhlmann showed that the linearized simultaneous recovery problem is unstable in any scale of Sobolev spaces \cite{StefanovUhlmann2013Instability}. Hence, in the general problem of recovering both $f$ and $c$, perturbations of the source and the sound speed may compensate for each other, making stable simultaneous recovery impossible without additional structure.

The purpose of this paper is to study a class of inverse source-speed problems
in which the initial source and the wave speed are linked by a prescribed
constitutive relation. More precisely, we assume that
\begin{equation}
\label{eq:intro_constitutive}
f(x)=\calF(x,c(x)),
\end{equation}
where $\calF$ is fixed by the admissible material class. This additional
structure reduces the simultaneous recovery of $f$ and $c$ to the recovery of
a single unknown. The key structural assumption is the non-degeneracy condition
\begin{equation}
\label{eq:intro_nd}
|\partial_r\calF(x,r)|\ge m_0>0,
\end{equation}
for the relevant range of speeds and in the region where recovery is sought.
This assumption is intended to model situations in which the source-generation mechanism and the sound speed are governed by the same underlying material.

The results of this paper are primarily structural, as they identify a mathematically precise mechanism by which a constitutive constraint restores uniqueness and Lipschitz stability in a simultaneous inverse problem that is otherwise severely underdetermined. To illustrate how such a constraint may arise, we briefly discuss two calibrated material regimes in which both the source amplitude and the wave speed are governed by a common underlying parameter.\\

\noindent\emph{Calibrated cure monitoring.}
One concrete context in which such a one-parameter description is natural is
cure monitoring for thermosetting resins, adhesives, or composites. Here one
should distinguish the slow processing time of the cure from the acoustic time
variable in the wave equation. At a fixed stage of the cure, let
\(\theta(x)\) denote the local degree of cure. Ultrasonic methods are used in
this setting because wave speed and attenuation are sensitive to the evolving
viscoelastic properties of the curing material
\cite{LionettoMaffezzoli2013}. In particular, the sound velocity is related to
the storage modulus and density, and hence reflects changes in mechanical
stiffness during cure. The experiments reviewed in
\cite{LionettoMaffezzoli2013} show that the velocity curve is not flat
throughout the cure process: after gelation the sound velocity increases
rapidly and later approaches a plateau near vitrification; in the reported
experiments the total velocity increase during the transition is about
\(1200\) m/s. Thus, on a calibrated subinterval of the cure process away from
the initial flat region and the final saturation plateau, it is reasonable to
model the acoustic speed at that fixed cure stage as a locally invertible
response curve
\[
        c(x)=C(x,\theta(x)),
        \qquad
        0<C_0\le |\partial_\theta C(x,\theta)|\le C_1 .
\]

If, in addition, the external excitation mechanism is calibrated and controlled
so that the generated source amplitude is also governed by the same local cure
state, then one may write
\[
        f(x)=P(x,\theta(x)).
\]
For example, in a thermoelastic generation regime the source amplitude depends
on the local mechanical response of the material to the prescribed thermal
loading; see also the discussion of thermoelastic laser ultrasonics below. On a
cure interval where this calibrated source response is non-flat, say
\[
        |\partial_\theta P(x,\theta)|\ge p_0>0,
\]
the cure parameter can be eliminated. Writing \(\theta=\Theta(x,c)\), one gets
\[
        f(x)=P(x,\Theta(x,c(x)))=\calF(x,c(x)),
\]
and
\[
        \partial_r\calF(x,r)
        =
        \frac{\partial_\theta P(x,\theta)}
             {\partial_\theta C(x,\theta)},
        \qquad r=C(x,\theta).
\]
Therefore the non-degeneracy condition
\[
        |\partial_r\calF(x,r)|\ge m_0>0
\]
is a local identifiability condition on the calibrated response curves: on the
selected cure interval, the same material state must affect both the acoustic
speed and the generated source amplitude in a quantitatively nontrivial way.\\

\noindent\emph{Calibrated thermoelastic generation.}
Consider a calibrated laser-ultrasonic experiment on a known material class.
Suppose that the laser-induced thermal strain profile $\varepsilon_T(x)$ is
calibrated independently, for example from a reference experiment or from a
heat-transfer model, and is therefore treated as known in the acoustic inverse
problem.

In a scalar reduction of thermoelastic generation, the initial source amplitude
is proportional to the local thermoelastic stress:
$$
        f(x)=M(x)\varepsilon_T(x),
$$
where $M(x)$ is the relevant scalar elastic modulus. If the density $\rho(x)$
is known and $M(x)=\rho(x)c^2(x)$, then
$$
        f(x)=\rho(x)\varepsilon_T(x)c^2(x).
$$
Thus, with $q(x)=\rho(x)\varepsilon_T(x)$, we obtain
$$
        f(x)=q(x)c^2(x),
        \qquad
        \mathcal F(x,r)=q(x)r^2.
$$
Moreover, the nondegeneracy condition holds on any region where
$q(x)\ge q_0>0$ and $c(x)\ge c_->0$, since
$$
        \partial_r\mathcal F(x,r)=2q(x)r.
$$

Under the geometric hypotheses used for one-measurement source recovery, we prove that equality of the boundary measurements implies equality of both the sound speed and the initial source. More precisely, for compactly supported perturbations and full boundary data, if the metric $g_1=c_1^{-2}dx^2$ admits a strictly convex foliation sweeping the unknown region and the observation time is sufficiently large, then
\[
\Lam_{f_1,c_1}=\Lam_{f_2,c_2}
\quad\Longrightarrow\quad
c_1=c_2,
\qquad
f_1=f_2
\]
in the swept region. Under the corresponding full visibility condition for the geodesic flow, we prove the Lipschitz estimate
\begin{equation*}
\|c_1-c_2\|_{L^2(K)}
+\|f_1-f_2\|_{L^2(K)}
\le C\|\Lam_{f_1,c_1}-\Lam_{f_2,c_2}\|_{L^2((0,T)\times\partial\Om)}.
\end{equation*}
The estimate involves the measured boundary trace difference itself.

We also prove local partial-data uniqueness. Let $\Gamma\subset\partial\Om$ be the observed part of the boundary, and let $\mathcal D_\Gamma$ be the region swept by strictly convex hypersurfaces visible from $\Gamma$ in the appropriate geometric sense. If the constitutive non-degeneracy condition \eqref{eq:intro_nd} holds only in a neighborhood of $\mathcal D_\Gamma$, then equality of the measurements on $(0,T)\times\Gamma$ implies
\[
c_1=c_2,
\qquad
f_1=f_2
\quad\text{in }\mathcal D_\Gamma.
\]
Under a corresponding partial-data stability geometry, we also prove a Lipschitz stability estimate in compact visible subregions. The microlocal source stability step is the partial-data analogue of the argument used in Theorem 3.4 in \cite{StefanovUhlmann2013SourceSpeed}; the constitutive reduction then converts that source estimate into simultaneous stability for the pair $(f,c)$. No monotonicity of $\calF$ is required in the unswept part of the medium. This local formulation is important: the constitutive law may be reliable only in a particular material class or target region, while the surrounding medium may not satisfy the same relation.

The paper is organized as follows. Section~\ref{sec:formulation} develops the full-boundary theory, beginning with the setup and reduction to an inverse source problems and then proving collar uniqueness, global uniqueness, and Lipschitz stability. Section~\ref{sec:partial} treats partial boundary data, the local nature of the constitutive non-degeneracy, and partial-data Lipschitz stability under a microlocal visibility condition.

\section{Recovery from full boundary data}
\label{sec:formulation}
	
	In this section, we formulate the full-boundary inverse problem and reduce the simultaneous recovery of the initial source and sound speed to an inverse source problem. We then establish collar uniqueness, global uniqueness, and Lipschitz stability under the corresponding geometric assumptions.
\subsection{Setup and reduction to an inverse source problem}
\label{sec:full_comparison}
	
	Let $\Om\subset\R^n$, $n\ge2$, be a bounded domain with smooth boundary. We normalize the known exterior speed to be one:
	\begin{equation*}
c(x)=1\qquad\text{for }x\in\R^n\setminus\Om.
	\end{equation*}
	Fix a compact set $K\Subset\Om$. To focus on the interior inverse problem and avoid boundary determination issues, all comparisons in the paper are made inside the class
	\begin{equation}
		\label{eq:support_assumption}
		\supp(c_1-c_2)\subset K.
	\end{equation}
	
	Let $0<c_-<c_+$. The smooth admissible class consists of speeds satisfying
	\begin{equation*}
c\in C^\infty(\R^n),\qquad c_-\le c\le c_+,
		\qquad c=1\text{ in }\R^n\setminus\Om,
	\end{equation*}
	and lying in a fixed bounded set in a sufficiently high $C^N$ norm. We do not attempt to optimize $N$; the smoothness assumption is imposed only to match the regularity required by the Carleman estimates used in \cite{StefanovUhlmann2013SourceSpeed}.
	
	For a speed $c$, let
	\begin{equation*}
P_c=\partial_t^2-c^2(x)\Delta.
	\end{equation*}
	The principal symbol corresponds to the Riemannian metric
	\begin{equation*}
g=c^{-2}(x)dx^2.
	\end{equation*}
	Geometric notions below, such as distance, convexity, geodesics and unit sphere bundles, are understood with respect to this metric. If $E\subset\Om$ is compact, we write $S_gE$ for the restriction of the unit sphere bundle of $g$ to $E$. For $(x,\xi)\in S_gE$, $\gamma_{x,\xi}$ denotes the unit-speed geodesic with initial data $(x,\xi)$.
	
	Let
	\begin{equation*}
		\calF\in C^\infty(\overline\Om\times[c_-,c_+]).
	\end{equation*}
	The initial source corresponding to the speed $c$ is
	\begin{equation*}
f_c(x)=\calF(x,c(x))\qquad x\in\Om,
	\end{equation*}
	extended outside $\Om$ in a fixed known way. The comparison assumption \eqref{eq:support_assumption} ensures that
	\begin{equation*}
\supp(f_{c_1}-f_{c_2})\subset K.
	\end{equation*}
	
	The forward solution $u_c$ is defined by
	\begin{equation*}
\begin{cases}
			P_cu_c=0, & (t,x)\in(0,T)\times\R^n,\\
			u_c(0,x)=f_c(x),\\
			\partial_tu_c(0,x)=0.
		\end{cases}
	\end{equation*}
	The measurement map is
	\begin{equation*}
\Lam_{f_c,c}=u_c|_{(0,T)\times\partial\Om}.
	\end{equation*}

	\paragraph{Constitutive non-degeneracy.}
	We impose the following condition on the constitutive law in the recovery set $K$: there exists $m_0>0$ such that
	\begin{equation*}
|\partial_r\calF(x,r)|\ge m_0,
		\qquad x\in K,
		\quad r\in[c_-,c_+].
	\end{equation*}
	We shall refer to this as constitutive non-degeneracy on $K$.
	
	For each fixed $x\in K$, the derivative $\partial_r\calF(x,\cdot)$ has a fixed sign on $[c_-,c_+]$. Set
	\begin{equation*}
b_0:=\frac{m_0}{2c_+}>0.
	\end{equation*}
	Then the divided difference of $\calF$ with respect to $r^2$ is uniformly bounded below. Indeed, if $r_1,r_2\in[c_-,c_+]$, $r_1\ne r_2$, then
	\begin{equation}
		\label{eq:divided_difference_preliminary}
		\left|
		\frac{\calF(x,r_1)-\calF(x,r_2)}{r_1^2-r_2^2}
		\right|
		=
		\frac1{r_1+r_2}
		\left|
		\int_0^1 \partial_r\calF(x,r_2+s(r_1-r_2))\,ds
		\right|
		\ge b_0.
	\end{equation}

We first prove the structural facts needed to recover the pair $(f,c)$. Throughout, $c_1,c_2$ are smooth admissible speeds and
	\begin{equation}\label{w}
w=u_{c_1}-u_{c_2},
		\qquad
		\alpha=c_1^2-c_2^2.
	\end{equation}
	Notice that $\alpha$ is supported in $K$.
	
	\begin{lemma}
		\label{lem:comparison_factorization}
		The difference $w$ satisfies
		\begin{equation}
			\label{eq:w_equation}
			P_{c_1}w=\alpha(x)\Delta u_{c_2},
		\end{equation}
		with initial data
		\begin{equation}
			\label{eq:w_initial}
			w(0,x)=\calF(x,c_1(x))-\calF(x,c_2(x)),
			\qquad w_t(0,x)=0.
		\end{equation}
		Moreover, if
		\begin{equation}
			\label{eq:B_definition}
			B(x)=
			\begin{cases}
				\displaystyle
				\frac{\calF(x,c_1(x))-\calF(x,c_2(x))}{c_1^2(x)-c_2^2(x)},
				& c_1(x)\ne c_2(x),\\[2ex]
				\displaystyle
				\frac{\partial_r\calF(x,c_1(x))}{2c_1(x)},
				& c_1(x)=c_2(x),
			\end{cases}
		\end{equation}
		then $B$ is smooth on $K$,
		\begin{equation*}
			\calF(x,c_1(x))-\calF(x,c_2(x))=\alpha(x)B(x),
		\end{equation*}
		and
		\begin{equation*}
			|B(x)|\ge b_0\qquad x\in K,
		\end{equation*}
		where $b_0=m_0/(2c_+)$.
	\end{lemma}
	
	\begin{proof}
		Subtract the equations $P_{c_j}u_{c_j}=0$ and write the result with the background operator $P_{c_1}$. Since
		$$
		P_{c_1}u_{c_2}
		=(\partial_t^2-c_1^2\Delta)u_{c_2}
		=(c_2^2-c_1^2)\Delta u_{c_2},
		$$
		we get \eqref{eq:w_equation}. The initial conditions follow from the constitutive relation and the assumption $\partial_tu_{c_j}(0)=0$.
		
		For $c_1\ne c_2$, the factorization identity is the definition of $B$. At points where $c_1=c_2$, the second line in \eqref{eq:B_definition} is the continuous extension of the divided difference because
		$$
		c_1^2-c_2^2=(c_1+c_2)(c_1-c_2).
		$$
		The lower bound is precisely \eqref{eq:divided_difference_preliminary}. Smoothness follows from the smoothness of $\calF$ and the usual smooth extension of divided differences.
	\end{proof}
	
	The next step removes the nonzero initial displacement. This is the technical device that makes the constrained simultaneous source-speed problem fall into the source class in \cite{StefanovUhlmann2013SourceSpeed}.
	
	For the stability argument we need a uniform class for the coefficient multiplying the unknown spatial factor in the source equation obtained below. If $E\Subset\Om$, $b>0$, and $M>0$, let $\mathcal C(E,T,b,M)$ denote the set of functions $a(t,x)$ such that
\[
        \|a\|_{C^2([0,T];C(\overline\Om))}\le M,
        \qquad
        a_t(0,x)=0,
        \qquad
        |a(0,x)|\ge b\quad x\in E.
\]
The condition $a_t(0,\cdot)=0$ is the compatibility condition that permits an even $C^2$ extension in time. This class will be applied to the coefficient $A(t,x)$ in Lemma~\ref{lem:time_primitive} below.

\begin{lemma}
	\label{lem:time_primitive}
	Let
	\begin{equation}
		\label{eq:V_definition}
		V(t,x)=\int_0^t(t-s)w(s,x),ds,
	\end{equation}
	where $w$ is defined by \eqref{w}. Then
	$$
	V(0,x)=V_t(0,x)=0
	$$
	and
	\begin{equation}
		\label{eq:V_source_equation}
		P_{c_1}V=\alpha(x)A(t,x),
	\end{equation}
	where
	\begin{equation}
		\label{eq:A_definition}
		A(t,x)=B(x)+\int_0^t(t-s)\Delta u_{c_2}(s,x),ds.
	\end{equation}
	In particular,
	\begin{equation}
		\label{eq:A_initial_properties}
		A(0,x)=B(x),\qquad A_t(0,x)=0,
		\qquad |A(0,x)|\ge b_0>0\quad x\in K.
	\end{equation}
	
\end{lemma}

\begin{proof}
	The zero initial conditions for $V$ follow directly from
	\eqref{eq:V_definition}. Since
	$$
	V_{tt}(t,x)=w(t,x),
	\qquad
	\Delta V(t,x)=\int_0^t(t-s)\Delta w(s,x),ds,
	$$
	we have
	\begin{equation} \label{P_c}
	P_{c_1}V(t,x)
	=
	w(t,x)-c_1^2(x)\int_0^t(t-s)\Delta w(s,x),ds.
	\end{equation}
	On the other hand,
	$$
	\begin{aligned}
		\int_0^t(t-s)P_{c_1}w(s,x),ds
		&=
		\int_0^t(t-s)w_{tt}(s,x),ds
		-c_1^2(x)\int_0^t(t-s)\Delta w(s,x),ds.
	\end{aligned}
	$$
	Integrating by parts in the first term and using $w_t(0,x)=0$, we obtain
	$$
	\int_0^t(t-s)w_{tt}(s,x),ds
	=
	w(t,x)-w(0,x).
	$$
	Therefore,
		$$
	\begin{aligned}
		w(t,x)-c_1^2(x)\int_0^t(t-s)\Delta w(s,x),ds=w(0,x)+\int_0^t(t-s)P_{c_1}w(s,x),ds.
	\end{aligned}
	$$
	Combining this with \eqref{P_c} yields 

	$$
	P_{c_1}V(t,x)
	=
	w(0,x)+\int_0^t(t-s)P_{c_1}w(s,x),ds.
	$$
	
	By Lemma~\ref{lem:comparison_factorization},
	$$
	w(0,x)=\alpha(x)B(x),
	\qquad
	P_{c_1}w(t,x)=\alpha(x)\Delta u_{c_2}(t,x).
	$$
	Hence
	$$
	\begin{aligned}
		P_{c_1}V(t,x)
		&=
		\alpha(x)B(x)
		+
		\alpha(x)\int_0^t(t-s)\Delta u_{c_2}(s,x),ds  \\
		&=
		\alpha(x)
		\left(
		B(x)+\int_0^t(t-s)\Delta u_{c_2}(s,x),ds
		\right)\\
		&=
		\alpha(x)A(t,x).
	\end{aligned}
	$$
	
	The identities $A(0,x)=B(x)$ and $A_t(0,x)=0$ follow directly from
	\eqref{eq:A_definition}. The lower bound
	$|A(0,x)|\ge b_0$ on $K$ follows from Lemma~\ref{lem:comparison_factorization}.
	
\end{proof}

	Next we show how the measured data enter the equation for $V$.
	
	\begin{lemma}
		\label{lem:boundary_data}
		On $(0,T)\times\partial\Om$,
		\begin{equation}
			\label{eq:V_boundary_trace}
			V(t,\cdot)=\int_0^t(t-s)(\Lam_{f_1,c_1}-\Lam_{f_2,c_2})(s,\cdot)\,ds,
		\end{equation}
		and
		\begin{equation}
			\label{eq:Vtt_boundary_trace}
			\partial_t^2V=\Lam_{f_1,c_1}-\Lam_{f_2,c_2}.
		\end{equation}
		If $\Lam_{f_1,c_1}=\Lam_{f_2,c_2}$ on the full boundary cylinder, then $V=\pnu V=0$ on $(0,T)\times\partial\Om$.
	\end{lemma}
	
	\begin{proof}
	By the definition of $V$,
	\[
		V(t,x)=\int_0^t(t-s)w(s,x)\,ds .
	\]
	Differentiating in time gives
	\[
		V_t(t,x)=\int_0^t w(s,x)\,ds,
		\qquad
		V_{tt}(t,x)=w(t,x).
	\]
	Since $w=u_{c_1}-u_{c_2}$, restriction to $(0,T)\times\partial\Om$ gives
	\[
		V_t(t,\cdot)
		=
		\int_0^t
		(\Lam_{f_1,c_1}-\Lam_{f_2,c_2})(s,\cdot)\,ds
	\]
	and
	\[
		\partial_t^2V|_{(0,T)\times\partial\Om}
		=
		\Lam_{f_1,c_1}-\Lam_{f_2,c_2}.
	\]
	This proves the two trace identities.

	If $\Lam_{f_1,c_1}=\Lam_{f_2,c_2}$ on $(0,T)\times\partial\Om$, then $V=0$ on the boundary. Since $\alpha$ is supported in $K\Subset\Om$, the function $V$ solves the known homogeneous exterior wave equation with zero exterior initial data. The exterior Dirichlet-to-Neumann argument in Lemma 3.1 in \cite{StefanovUhlmann2013SourceSpeed} then gives
	\[
		\partial_\nu V=0
		\qquad\text{on }(0,T)\times\partial\Om .
	\]
	Hence $V=\partial_\nu V=0$ on the full boundary cylinder.
\end{proof}

\subsection{Collar uniqueness}
\label{sec:collar_uniqueness}
	
	The first full-boundary result gives uniqueness in a boundary collar, that is, in the region reached from $\partial\Om$ by the $g_1$-normal distance function before time $T$. It is the simultaneous source-speed analogue of the collar source uniqueness result in Theorem 3.1 in \cite{StefanovUhlmann2013SourceSpeed}.

\begin{theorem}
		\label{thm:collar_uniqueness}
		Let $c_1,c_2$ be smooth admissible speeds satisfying \eqref{eq:support_assumption}, and assume constitutive non-degeneracy on $K$. Set $f_j=\calF(\cdot,c_j)$, $j=1,2$. Let $g_1=c_1^{-2}dx^2$. Suppose $\partial\Om$ is strictly convex with respect to $g_1$, $x^n=\dist_{g_1}(x,\partial\Om)$ is smooth with nonzero differential for $0\le x^n\le T$, and the hypersurfaces $\{x^n=s\}$ are strictly convex for $0\le s<T$. If
		\begin{equation}
			\label{eq:full_data_equal_main}
			\Lam_{f_1,c_1}=\Lam_{f_2,c_2}
			\qquad\text{on }(0,T)\times\partial\Om,
		\end{equation}
		then
		\begin{equation*}
			c_1=c_2,
			\qquad
			f_1=f_2
			\qquad\text{on }K\cap\{x\in\Om:\dist_{g_1}(x,\partial\Om)<T\}.
		\end{equation*}
	\end{theorem}

\begin{proof}
		The preceding lemmas put the comparison problem in the precise inverse source form required in \cite{StefanovUhlmann2013SourceSpeed}. Namely, by Lemma~\ref{lem:time_primitive}, the time primitive $V$ satisfies the source equation \eqref{eq:V_source_equation}; explicitly,
		$$
		P_{c_1}V=\alpha(x)A(t,x),
		\qquad V|_{t=0}=V_t|_{t=0}=0,
		$$
		where $\alpha=c_1^2-c_2^2$ is supported in $K$ and $|A(0,x)|\ge b_0$ on $K$. By Lemma~\ref{lem:boundary_data}, equality of the measurements gives zero Cauchy data for $V$ on the boundary. Extending $V$ and $A$ evenly in time, we apply Theorem 3.1 in \cite{StefanovUhlmann2013SourceSpeed} with background speed $c_1$. Thus
		$$
		\alpha=0
		\qquad\text{for }\dist_{g_1}(x,\partial\Om)<T.
		$$
		Since $c_1,c_2>0$, the equality $c_1^2=c_2^2$ implies $c_1=c_2$ in the collar. Because $f_j=\calF(\cdot,c_j)$, we also have $f_1=f_2$ there. This proves simultaneous recovery of the speed and the initial source in the stated set.
	\end{proof}

\subsection{Global uniqueness from full data}
\label{sec:global_full_uniqueness}
	
	We next formulate the global full-boundary foliation condition. Let $\{\Sigma_s:s_1\le s\le s_2\}$ be a continuous family of smooth compact oriented hypersurfaces intersecting $\overline\Om$. Each $\Sigma_s$ is given an interior side $\Sigma_s^{\rm int}$ and an exterior side $\Sigma_s^{\rm ext}$, where the exterior side is the side from which the foliation propagates toward the unknown region.
	
	\paragraph{Full-boundary convex foliation.}
	We say that $K$ satisfies the full-boundary convex foliation condition, with respect to $g_1=c_1^{-2}dx^2$ and the observation time $T$, if there is such a family $\{\Sigma_s:s_1\le s\le s_2\}$ satisfying
	\begin{enumerate}[label=(\roman*)]
		\item each $\Sigma_s\cap\overline\Om$ is strictly convex with respect to $g_1$;
		\item the family sweeps $K$, in the sense that every point of $K$ lies on one of the hypersurfaces, and $K\subset\Sigma_{s_1}^{\rm int}$;
		\item for every $s$ and every $x\in\Sigma_s\cap\overline\Om$ there is $y\in\partial\Om$ such that
		$$
		T>\dist_{g_1}(x,y).
		$$
	\end{enumerate}
	A sufficient condition for (iii) is
	\begin{equation*}
T>\max_s\dist_{g_1}(\Sigma_s\cap\overline\Om,\partial\Om).
	\end{equation*}
	This is the full-boundary convex foliation geometry used for source uniqueness in Theorem 2.3 in \cite{StefanovUhlmann2013SourceSpeed}.

\begin{theorem}
		\label{thm:global_uniqueness}
		Let $c_1,c_2$ be smooth admissible speeds satisfying \eqref{eq:support_assumption}, and set $f_j=\calF(\cdot,c_j)$, $j=1,2$. Assume constitutive non-degeneracy on $K$ and suppose that $K$ satisfies the full-boundary convex foliation condition with respect to $g_1$ and $T$. If \eqref{eq:full_data_equal_main} holds, then
		$$
		c_1=c_2,
		\qquad
		f_1=f_2
		\qquad\text{in }K.
		$$
		Since $\supp(c_1-c_2)\subset K$, the equality $c_1=c_2$ also holds in
$\Om\setminus K$. Hence, by the constitutive relation, $f_1=f_2$ in
$\Om\setminus K$ as well. Therefore
$$
        c_1=c_2,
        \qquad
        f_1=f_2
        \qquad\text{in }\Om .
$$
	\end{theorem}

\begin{proof}
		By Lemma~\ref{lem:time_primitive}, the time primitive $V$ satisfies the source equation \eqref{eq:V_source_equation}; explicitly,
		\begin{equation*}
			P_{c_1}V=\alpha(x)A(t,x),
			\qquad V|_{t=0}=V_t|_{t=0}=0,
			\qquad \alpha=c_1^2-c_2^2.
		\end{equation*}
		Moreover, by \eqref{eq:A_initial_properties},
		\begin{equation*}
			A_t(0,x)=0,
			\qquad |A(0,x)|\ge b_0\quad x\in K.
		\end{equation*}
		The condition $A_t(0,\cdot)=0$ permits the standard even extension in time used in the source recovery theorem.
		
		By Lemma~\ref{lem:boundary_data}, the equality of the full boundary measurements in \eqref{eq:full_data_equal_main} gives
		\begin{equation*}
			V=\partial_\nu V=0
			\qquad\text{on }(0,T)\times\partial\Om .
		\end{equation*}
		The hypotheses of Theorem 2.3 in \cite{StefanovUhlmann2013SourceSpeed} are therefore satisfied for the background speed $c_1$, the spatial source factor $\alpha$, the coefficient $A$, and the foliation $\{\Sigma_s\}$. Applying that theorem gives
		\begin{equation*}
			\alpha=0
			\qquad\text{on the region swept by the foliation}.
		\end{equation*}
		In particular, $\alpha=0$ on $K$. Since $\alpha=c_1^2-c_2^2$ and the speeds are positive, this implies $c_1=c_2$ in $K$. The constitutive relation then gives $f_1=f_2$ in $K$. Finally, the admissible class assumption \eqref{eq:support_assumption} gives $c_1=c_2$ in $\Om\setminus K$, and the constitutive relation gives $f_1=f_2$ there as well. Hence
		\begin{equation*}
			c_1=c_2,
			\qquad
			f_1=f_2
			\qquad\text{in }\Om .
		\end{equation*}
	\end{proof}

\subsection{Lipschitz stability with full boundary data}
\label{sec:full_stability}
	
	For stability one needs a stronger condition than the reachability requirement in the full-boundary foliation condition. The precise condition is the following geodesic visibility condition, which is condition (51) in Theorem 3.4 in \cite{StefanovUhlmann2013SourceSpeed}.
	
	\paragraph{Full-boundary visibility.}
	We say that $K$ satisfies the full-boundary visibility condition, with respect to $g_1$ and the observation time $T$, if every $(x,\xi)\in S_{g_1}K$ reaches the boundary in forward or backward time less than $T$:
	\begin{equation}
		\label{eq:visibility_main}
		\forall (x,\xi)\in S_{g_1}K,
		\quad \exists t\in(-T,T)
		\quad\text{such that}\quad
		\gamma_{x,\xi}(t)\in\partial\Om.
	\end{equation}

\begin{theorem}
		\label{thm:lipschitz_stability}
		Let $c_1,c_2$ lie in a bounded smooth admissible class satisfying \eqref{eq:support_assumption}, set $f_j=\calF(\cdot,c_j)$, $j=1,2$, and suppose constitutive non-degeneracy holds uniformly on $K$. Let $g_1=c_1^{-2}dx^2$. Assume that $K$ satisfies the full-boundary convex foliation condition as well as the full-boundary visibility condition with respect to $g_1$ and $T$. Then
		\begin{equation}
			\label{eq:main_stability}
			\|c_1-c_2\|_{L^2(K)}
			+\|f_1-f_2\|_{L^2(K)}
			\le C\|\Lam_{f_1,c_1}-\Lam_{f_2,c_2}\|_{L^2((0,T)\times\partial\Om)}.
		\end{equation}
		The constant $C$ depends only on the admissible class, the lower bound in the constitutive non-degeneracy condition, the visibility geometry, $K$, and $T$.
	\end{theorem}

\begin{proof} By Lemma~\ref{lem:time_primitive}, the time primitive $V$ satisfies the source equation \eqref{eq:V_source_equation}, together with the zero initial conditions
		$$
		V|_{t=0}=V_t|_{t=0}=0,
		\qquad
		\alpha=c_1^2-c_2^2.
		$$
		Moreover, the a priori smooth bounds imply that
$$
A\in\mathcal C(K,T,b_0,M_A)
$$
for some $M_A$ depending on the admissible class, $T$, and $\calF$.
		Since $c_1$ is fixed, Theorem~3.4 in
\cite{StefanovUhlmann2013SourceSpeed}, applied with background
speed $c_1$, spatial source factor $\alpha$, and coefficient $A$,
gives
\[
\|\alpha\|_{L^2(K)}
\le C\|\partial_t^2V\|_{L^2((0,T)\times\partial\Om)},
\]
where $C$ is uniform as $c_2$ varies in the admissible class.

		By Lemma~\ref{lem:boundary_data},
		$$
		\partial_t^2V|_{(0,T)\times\partial\Om}
		=
		\Lam_{f_1,c_1}-\Lam_{f_2,c_2}.
		$$
		Therefore
		$$
		\|c_1^2-c_2^2\|_{L^2(K)}
		\le
		C\|\Lam_{f_1,c_1}-\Lam_{f_2,c_2}\|_{L^2((0,T)\times\partial\Om)}.
		$$
		Since $c_j\ge c_->0$,
		$$
		|c_1-c_2|
		\le
		\frac{1}{2c_-}|c_1^2-c_2^2|.
		$$
		Hence
		$$
		\|c_1-c_2\|_{L^2(K)}
		\le
		C\|\Lam_{f_1,c_1}-\Lam_{f_2,c_2}\|_{L^2((0,T)\times\partial\Om)}.
		$$
		Finally, since $f_j=\calF(\cdot,c_j)$ and $\calF$ is uniformly Lipschitz in the sound-speed variable on the admissible range,
		$$
		|f_1(x)-f_2(x)|
		=
		|\calF(x,c_1(x))-\calF(x,c_2(x))|
		\le C|c_1(x)-c_2(x)|.
		$$
		Combining the last two estimates proves \eqref{eq:main_stability}. The constant $C$ may depend on $c_1$ and its visibility geometry,
but is uniform with respect to $c_2$ in the admissible class.
	\end{proof}

\begin{remark}
		The stability estimate uses the measured boundary trace difference itself. Although the source stability estimate in Theorem 3.4 in \cite{StefanovUhlmann2013SourceSpeed} contains $\partial_t^2V$ on the boundary, the construction gives
		$$
		\partial_t^2V|_{(0,T)\times\partial\Om}
		=
		\Lam_{f_1,c_1}-\Lam_{f_2,c_2}.
		$$
		Thus no differentiation of the measured data is lost in the final estimate.
	\end{remark}

\section{Recovery from partial boundary data}
\label{sec:partial}
	
	We now state the partial-data assumptions explicitly. Let $\Gamma\subset\partial\Om$ be relatively open and let
	\begin{equation*}
G=\{(t,y):y\in\Gamma,\ 0<t<\tau(y)\},
	\end{equation*}
	where $\tau$ is positive and continuous. Here $G$ is the open observation set used in the data norm. In the exterior cone condition below we also allow the initial time slice $t=0$, where the exterior solution has zero Cauchy data. In partial data, the exterior Dirichlet trace does not determine the Neumann trace at every point of $G$. The following exterior cone condition is used to identify those boundary points for which the needed Cauchy data are determined by the measured Dirichlet trace and the known exterior equation, as in Theorem 3.2 in \cite{StefanovUhlmann2013SourceSpeed}.
	
	For $(\hat t,\hat y)\in\R_+\times\partial\Om$, define
	\begin{equation*}
C_{\hat t,\hat y}
		=\{(t,y)\in\R_+\times\partial\Om:
		t+\dist_0(y,\hat y)\le \hat t\},
	\end{equation*}
	where $\dist_0$ is the distance in the known exterior $\R^n\setminus\Om$.
	
		Let $\{\Sigma_s:s_1\le s\le s_2\}$ be a continuous family of smooth compact oriented hypersurfaces intersecting $\overline\Om$. Each $\Sigma_s$ is given an interior side $\Sigma_s^{\rm int}$ and an exterior side $\Sigma_s^{\rm ext}$, where the exterior side is the side from which the partial-data propagation starts.
	
	\paragraph{Partial observation geometry.}
	We say that the family $\{\Sigma_s:s_1\le s\le s_2\}$ is an admissible partial-data convex foliation for the observation set $G$, with respect to $g_1=c_1^{-2}dx^2$, if
	\begin{enumerate}[label=(\roman*)]
		\item each $\Sigma_s\cap\overline\Om$ is strictly convex with respect to $g_1$;
		\item $\Sigma_s\cap(\partial\Om\setminus\Gamma)=\emptyset$ for every $s$.
	\end{enumerate}
	For such a family, define the recoverable region
	\begin{equation}
		\label{eq:D_partial}
		\mathcal D_\Gamma
		=\Big\{x\in\Om\cap\bigcup_s\Sigma_s:
		\exists y\in\Gamma\text{ such that }
		C_{\dist_{g_1}(x,y),y}
		\subset \{(t,z):z\in\Gamma,\ 0\le t<\tau(z)\}
		\Big\}.
	\end{equation}
	This is the partial-data foliation and cone geometry used for source uniqueness in Theorem 3.2 in \cite{StefanovUhlmann2013SourceSpeed}.
	
	The following statement is intentionally local. The constitutive law need not be non-degenerate in the whole domain; it is enough that it be non-degenerate on the portion of the medium through which the Carleman propagation is carried out.
	
	\begin{theorem}
		\label{thm:partial_uniqueness}
		Let $c_1,c_2$ be smooth admissible speeds, set $\alpha=c_1^2-c_2^2$, and set $f_j=\calF(\cdot,c_j)$, $j=1,2$. Assume that $\mathcal D_\Gamma$ is defined by an admissible partial-data convex foliation and that the support-side condition
		\begin{equation}
			\label{eq:partial_uniqueness_support}
			\supp\alpha\subset \Sigma_{s_1}^{\rm int}.
		\end{equation}
		Suppose that the constitutive law is non-degenerate only in a neighborhood of the swept recoverable region: there exists an open set $U\Subset\Om$ with $\overline{\mathcal D_\Gamma}\subset U$ and a constant $m_U>0$ such that
		\begin{equation}
			\label{eq:local_constitutive_nd}
			|\partial_r\calF(x,r)|\ge m_U,
			\qquad x\in U,
			\quad r\in[c_-,c_+].
		\end{equation}
		If
		$$
		\Lam_{f_1,c_1}=\Lam_{f_2,c_2}\qquad\text{on }G,
		$$
		then
		$$
		c_1=c_2,
		\qquad
		f_1=f_2
		\qquad\text{in }\mathcal D_\Gamma.
		$$
	\end{theorem}
	
	\begin{proof}
		The factorization in Lemma~\ref{lem:comparison_factorization} is pointwise. Under \eqref{eq:local_constitutive_nd}, the divided-difference coefficient satisfies
		$$
		|B(x)|\ge \frac{m_U}{2c_+}\qquad x\in U.
		$$
		Thus, in the source equation \eqref{eq:V_source_equation}, the coefficient $A(t,x)$ satisfies the required non-degeneracy condition at $t=0$ at every point of the swept recoverable region. No lower bound for $A(0,x)$ is needed outside $U$, since the source recovery argument is used only while the convex surfaces pass through $\mathcal D_\Gamma$.
		
		Equality of the partial Dirichlet data gives $V=0$ on $G$. The cone condition in \eqref{eq:D_partial}, with the initial time slice included as indicated in the definition of $\mathcal D_\Gamma$, gives the corresponding vanishing of the Neumann trace at the boundary points used for the Carleman propagation, by the exterior finite-propagation argument in the known exterior. Applying Theorem 3.2 in \cite{StefanovUhlmann2013SourceSpeed} locally along the family $\{\Sigma_s\}$ gives $\alpha=0$ in $\mathcal D_\Gamma$. Since the speeds are positive, this implies $c_1=c_2$ in $\mathcal D_\Gamma$. The constitutive relation then gives $f_1=f_2$ in $\mathcal D_\Gamma$. Hence the partial data determine the pair $(f,c)$ in the recoverable region.
	\end{proof}
	
	\begin{remark}[Locality of the constitutive assumption]
		\label{rem:local_constitutive}
		The preceding theorem does not require monotonicity of $r\mapsto\calF(x,r)$ throughout $\Om$. It requires monotonicity only on the part of the medium swept by the recoverable strictly convex surfaces, or more precisely on a neighborhood of that set. The speed and the initial source may remain unknown in the unswept region, and the constitutive law may even fail to be monotone there. What is essential is that the propagation of the Carleman argument from the observed boundary to the target set does not have to cross a region where the divided difference
		$$
		\frac{\calF(x,c_1)-\calF(x,c_2)}{c_1^2-c_2^2}
		$$
		can vanish. Thus the material non-degeneracy is a local requirement along the recoverable path, not a global structural condition on the whole object.
	\end{remark}
	
	For partial-data stability we impose a microlocal visibility condition together
with the exterior cone condition needed to recover Neumann data from the observed
Dirichlet trace. This is the partial-data analogue of the visibility condition used
in Theorem 3.4 in \cite{StefanovUhlmann2013SourceSpeed}, with the cone geometry of
Theorem 3.2 there.

\paragraph{Partial-data stability geometry.}
Let $K_0\Subset\mathcal D_\Gamma$. We say that $K_0$ satisfies the partial-data
stability geometry with respect to $g_1=c_1^{-2}dx^2$ and the observation set $G$
if there exist a relatively open set $\Gamma_1$ with $\overline{\Gamma_1}\subset\Gamma$ and a positive
continuous function $\tau_1$ on $\overline{\Gamma_1}$, with $\tau_1<\tau$ on $\overline{\Gamma_1}$, such
that the inner observation window
\[
        G_1=\{(t,y):y\in\Gamma_1,\ 0\le t<\tau_1(y)\}
\]
has the following two properties. First, with the same convex foliation used to define
$\mathcal D_\Gamma$, the inner observation window $G_1$
satisfies
\[
K_0\Subset\mathcal D_{\Gamma_1},
\qquad
K_0\subset\Sigma_{s_1}^{\rm int},
\]
where $\mathcal D_{\Gamma_1}$ is defined by
\eqref{eq:D_partial} with $(\Gamma,\tau)$ replaced by
$(\Gamma_1,\tau_1)$. Second,
for every $(x,\xi)\in S_{g_1}K_0$, at least one of the two unit-speed geodesics
$\gamma_{x,\xi}$ and $\gamma_{x,-\xi}$ reaches a point $y\in\Gamma_1$ at a time
$t$ with $0<t<\tau_1(y)$, and the arrival is transversal:
\[
        \gamma(t)=y,
        \qquad
        \dot\gamma(t)\notin T_y(\partial\Om).
\]
The transversality condition excludes glancing arrivals. The use of the smaller
window $G_1$ ensures that the arrivals used in the stability argument are separated
from the boundary of the observed set, while the inclusion of the initial time
slice in $G_1$ is used only for the exterior cone argument and for integrating
$\partial_t^2 Z$ in time.

\begin{proposition}
\label{prop:partial_source_stability}
Let $K_0\Subset\mathcal D_\Gamma$ be compact, and assume that $K_0$
satisfies the partial-data stability geometry with respect to
$g_1=c_1^{-2}dx^2$ and $G$. Let $\calA$ be a class of smooth coefficients
uniformly bounded in $C^N([0,T]\times\overline\Om)$, where $N$ is fixed
and sufficiently large, and assume that, for every $a\in\calA$,
$$
a_t(0,\cdot)=0,
\qquad
|a(0,x)|\geq b>0
\quad\text{for all }x\in K_0.
$$
Then there exists a constant $C>0$, independent of $a\in\calA$, such that,
for every $h\in L^2(\Om)$ with $\supp h\subset K_0$, the solution $Z$ of
\begin{equation}
\label{eq:partial_source_problem}
P_{c_1}Z=a(t,x)h(x),
\qquad
Z|_{t=0}=Z_t|_{t=0}=0,
\end{equation}
satisfies
\begin{equation}
\label{eq:partial_source_stability}
\|h\|_{L^2(\Om)}
\leq
C\|\partial_t^2Z\|_{L^2(G)}.
\end{equation}
\end{proposition}

\begin{proof}
The proof combines the partial-data microlocal stability argument in
\cite[Theorem~3]{StefanovUhlmann2009} with the Duhamel and Fredholm
argument in \cite[Theorem~3.4]{StefanovUhlmann2013SourceSpeed}.
The partial-data uniqueness result
\cite[Theorem~3.2]{StefanovUhlmann2013SourceSpeed}
is used to remove the compact remainder.

Choose $\chi\in C^\infty([0,T]\times\partial\Om)$ such that
\[
0\leq\chi\leq1,
\qquad
\chi=1
\quad\text{in a neighborhood of }G_1,
\]
and
\[
\supp\chi
\subset
\{(t,y):y\in\Gamma,\ 0\leq t<\tau(y)\}.
\]
It is enough to prove the estimate with
$\|\chi\,\partial_t^2Z\|_{L^2(G)}$ on the right-hand side.

Set $W=\partial_t^2Z$. Differentiating the equation twice in time and using
the zero initial conditions for $Z$, together with $a_t(0,\cdot)=0$, gives
\[
P_{c_1}W=a_{tt}(t,x)h(x),
\qquad
W|_{t=0}=a(0,x)h(x),
\qquad
W_t|_{t=0}=0.
\]
The Duhamel formula separates $W$ into the homogeneous wave with initial
displacement $a(0,\cdot)h$ and the term generated by $a_{tt}h$.

We now localize the boundary back-projection microlocally. By the
partial-data stability geometry, for every
$(x,\xi)\in S_{g_1}K_0$, at least one of the two geodesic branches issued
from $(x,\xi)$ reaches $G_1$ transversally. By compactness of
$S_{g_1}K_0$, one may choose finitely many conic open sets
$\mathcal U_j\subset T^*K_0\setminus0$, together with corresponding
microlocal cutoffs on $T^*G$, such that on each $\mathcal U_j$ one
transversal branch reaching $G_1$ is selected. Applying the boundary
back-projection construction from the proof of
\cite[Theorem~3.4]{StefanovUhlmann2013SourceSpeed} to these localized data
and summing the resulting operators produces an order-zero
pseudodifferential operator on $K_0$.

Its principal symbol is the sum of the contributions of the selected
visible branches, multiplied by $a(0,x)/2$. Since every
$(x,\xi)\in S_{g_1}K_0$ belongs to at least one of the conic sets on which
a selected branch reaches $G_1$, where $\chi=1$, and since
\[
|a(0,x)|\geq b
\qquad\text{for }x\in K_0,
\]
the absolute value of the principal symbol is uniformly bounded below on
$T^*K_0\setminus0$. The resulting operator is therefore elliptic.

The corresponding parametrix construction gives
\[
h
=
R_a\bigl(\chi\,\partial_t^2Z|_G\bigr)+K_a h
\qquad\text{on }K_0,
\]
where
\[
R_a:L^2(G)\longrightarrow L^2(K_0)
\]
is bounded and
\[
K_a:L^2(K_0)\longrightarrow L^2(K_0)
\]
is compact. As in \cite{StefanovUhlmann2013SourceSpeed}, the compact
remainder consists of the negative-order parametrix error and the
back-projected Duhamel term.

It remains to show that the corresponding kernel is trivial. Suppose that
\[
\chi\,\partial_t^2Z|_G=0.
\]
Since $\chi=1$ on $G_1$, we have $\partial_t^2Z=0$ on $G_1$. The time
fibers of $G_1$ begin at $t=0$, and
\[
Z|_{t=0}=Z_t|_{t=0}=0;
\]
hence $Z=0$ on $G_1$. The exterior cone argument then gives the
corresponding vanishing of the Neumann trace at the boundary points used
in the partial-data propagation. Moreover, $a_t(0,\cdot)=0$ permits the
required even extension across $t=0$.

We may therefore apply the localized source uniqueness argument from the
proof of \cite[Theorem~3.2]{StefanovUhlmann2013SourceSpeed}. Because
$\supp h\subset K_0$, the nonvanishing condition on $a(0,\cdot)$ is needed
only on $K_0$: outside $K_0$ the spatial source factor already vanishes.
The convex-surface propagation and the exterior cone argument consequently
give
\[
h=0.
\]
Thus the kernel is trivial, and the compact term is removed by the usual
Fredholm argument. For each fixed $a$, we obtain
\[
\|h\|_{L^2(\Om)}
\leq
C_a\|\chi\,\partial_t^2Z\|_{L^2(G)}.
\]

Finally, the localized observation and back-projection operators depend
continuously on $a$ in the $C^2$ topology. Hence the estimate persists,
with a comparable constant, in a sufficiently small $C^2$ neighborhood
of each fixed coefficient. Since $\calA$ is uniformly bounded in a
sufficiently high $C^N$ norm, it is precompact in $C^2$. A finite covering
argument therefore makes the constant uniform for $a\in\calA$. Since
\[
\|\chi\,\partial_t^2Z\|_{L^2(G)}
\leq
\|\partial_t^2Z\|_{L^2(G)},
\]
the result follows.
\end{proof}

	\begin{theorem}
		\label{thm:partial_stability}
		Let $K_0\Subset\mathcal D_\Gamma$ be compact and assume that $K_0$ satisfies the
partial-data stability geometry with respect to $g_1=c_1^{-2}dx^2$ and $G$. Let
$c_1,c_2$ lie in a bounded smooth admissible class, set $\alpha=c_1^2-c_2^2$, and
set $f_j=\calF(\cdot,c_j)$, $j=1,2$. Assume
		\begin{equation}
			\label{eq:partial_stability_support}
			\supp(c_1-c_2)\subset K_0.
		\end{equation}
		Assume also that the local constitutive non-degeneracy condition
		$$
		|\partial_r\calF(x,r)|\ge m_U>0,
		\qquad x\in U,
		\quad r\in[c_-,c_+],
		$$
		holds in an open set $U$ containing $K_0$. Then
		\begin{equation}
			\label{eq:partial_stability}
			\|c_1-c_2\|_{L^2(K_0)}
			+
			\|f_1-f_2\|_{L^2(K_0)}
			\le
			C\|\Lam_{f_1,c_1}-\Lam_{f_2,c_2}\|_{L^2(G)}.
		\end{equation}
		The constant $C$ depends only on the source stability constants in Proposition~\ref{prop:partial_source_stability}, the admissible class, the lower speed bound $c_-$, and the Lipschitz norm of $\calF$ on the admissible range.
	\end{theorem}

	\begin{proof} Set
		$$
		b_U:=\frac{m_U}{2c_+}.
		$$
		By Lemma~\ref{lem:time_primitive}, the second time primitive $V$ satisfies the source equation \eqref{eq:V_source_equation}; explicitly,
		$$
		P_{c_1}V=\alpha(x)A(t,x),
		\qquad
		V|_{t=0}=V_t|_{t=0}=0,
		\qquad
		\alpha=c_1^2-c_2^2.
		$$
		The support condition \eqref{eq:partial_stability_support} gives $\supp\alpha\subset K_0$. By the local constitutive non-degeneracy on $U$ and the divided-difference argument in Lemma~\ref{lem:comparison_factorization},
		$$
		|A(0,x)|\ge b_U>0,
		\qquad x\in K_0.
		$$
		Moreover, the a priori smooth bounds on the admissible class give
$A\in\mathcal C(K_0,T,b_U,M_A)$ for some $M_A$ depending only on the admissible
class, $T$, and $\calF$. For the fixed background speed $c_1$, as $c_2$ varies in the
bounded smooth admissible class, the corresponding coefficients
$A$ form a bounded smooth coefficient class of the type required
in Proposition~\ref{prop:partial_source_stability}.

		Applying Proposition~\ref{prop:partial_source_stability} to $Z=V$, $h=\alpha$, and $a=A$, we obtain
		$$
		\|\alpha\|_{L^2(K_0)}
		\le
		C\|\partial_t^2V\|_{L^2(G)}.
		$$
		By Lemma~\ref{lem:boundary_data},
		$$
		\partial_t^2V|_G
		=
		\Lam_{f_1,c_1}-\Lam_{f_2,c_2}.
		$$
		Thus
		$$ 
		\|c_1^2-c_2^2\|_{L^2(K_0)}
		\le
		C\|\Lam_{f_1,c_1}-\Lam_{f_2,c_2}\|_{L^2(G)}.
		$$
		Since $c_j\ge c_->0$,
		$$
		|c_1-c_2|
		\le
		\frac{1}{2c_-}|c_1^2-c_2^2|.
		$$
		Finally, since $f_j=\calF(\cdot,c_j)$ and $\calF$ is uniformly Lipschitz in the sound-speed variable on the admissible range,
		$$
		|f_1(x)-f_2(x)|
		\le C|c_1(x)-c_2(x)|.
		$$
		Combining the last estimates proves \eqref{eq:partial_stability}.
	\end{proof}
	
	\section*{Acknowledgments}

The author used ChatGPT (OpenAI) to assist with language editing.

\end{document}